\documentclass[11pt, leqno]{article}
\usepackage{amsmath,amssymb,amscd,amsthm, latexsym}

\newtheorem{theorem}{Theorem}[section]
\newtheorem{lemma}[theorem]{Lemma}
\newtheorem{prop}[theorem]{Proposition}
\newtheorem{cor}[theorem]{Corollary}

\newtheorem{conjecture}[theorem]{Conjecture}

\theoremstyle{definition}
\newtheorem{defn}[theorem]{Definition}

\theoremstyle{remark}

\numberwithin{equation}{section}

\begin{document}

\title%[Solvable automorphism groups]
  {\bf Solvable automorphism groups of a compact K\" ahler manifold}

%%submitted version

\author{Jin Hong Kim}%%\\
       %%Department of Mathematical Sciences\\
       %%Korea Advanced Institute of Science and Technology\\
       %%Yusong-gu, Daejon 305--701, Republic of Korea\\
       %%E-mail: jinkim@kaist.ac.kr
       %%}

%%\date{%% August 27, 2007:
%%\today}

%%\keywords{}

%%\subjclass{Primary }

%%\email{jinkim@math.kaist.ac.kr}

%%\thanks{This work was supported by the Korea Research Foundation Grant.}

\maketitle

\begin{abstract}
Let $X$ be a compact K\" ahler manifold of complex dimension $n$. Let $G$ be a connected solvable subgroup of the automorphism group ${\rm Aut}(X)$, and let $N(G)$ be the normal subgroup of $G$ of elements of null entropy. One of the goals of this paper is to show that $G/N(G)$ is a free abelian group of rank $r(G) \le n-1$ and that the rank estimate is optimal. This gives an alternative proof of the conjecture of Tits type in \cite{KOZ}. In addition, we show some non-obvious implications on the structure of solvable automorphism groups of compact K\" ahler manifolds. Furthermore, we also show that if the rank $r(G)$ of the quotient group $G/N(G)$ is equal to $n-1$ and the identity component ${\rm Aut}_0(X)$ of ${\rm Aut}(X)$ is trivial, then $N(G)$ is a finite set. The main strategy of this paper is to combine the method of Dinh and Sibony and the theorem of Birkhoff-Perron-Frobenius (or Lie-Kolchin type), and one argument of D.-Q. Zhang originated from the paper of Dinh and Sibony plays an important role.
\end{abstract}

\section{Introduction and Main Results} \label{sec1} %section 1

Let $X$ be a compact K\" ahler manifold. We denote by ${\rm Aut}(X)$ the biholomorphism group of $X$.  The aim of this paper is to study the structure of the automorphism group ${\rm Aut}(X)$ of $X$. In fact, in \cite{DS} Dinh and Sibony have already given interesting results for abelian automorphism groups of compact K\" ahler manifolds.
To describe our main theorem, we first need to set up some terminology. Let $g$ be an automorphism of $X$. The spectral radius
\[
\rho(g)=\rho(g^\ast|_{H^2(X, {\bf C})})
\]
of the action of $g$ on the cohomology ring $H^\ast(X, {\bf C})$ is defined to be the maximum of the absolute values of eigenvalues on the ${\bf C}$-linear extension of $g^\ast|_{H^2(X, {\bf C})}$. It is obvious from the definition that $\rho(g^{\pm 1})$ is always less than or equal to $\rho(g^{\mp})^{n-1}$ (e.g., see \cite{Gu05}). We call $g$ \emph{of null entropy} (resp. \emph{of positive entropy}) if the spectral radius $\rho(g)$ is equal to $1$ (resp. $>1$). It is known by the results of Gromov and Yomdin in \cite{Gr} and \cite{Yo} that
\[
\rho(g^\ast|_{H^2(X, {\bf C})})=\rho(g^\ast|_{H^{1,1}(X, {\bf C})}).
\]
We say that a subgroup $G$ of automorphisms is \emph{of null entropy} (resp. \emph{of positive entropy}) if all nontrivial elements of $G$ are of null entropy (resp. of positive entropy).

With these said, Dinh and Sibony proved in \cite{DS} that if $G$ is a abelian subgroup of ${\rm Aut}(X)$ and $N(G)$ is the set of elements of zero entropy, then $N(G)$ is a normal subgroup of $G$ and $G/N(G)$ is a free abelian group of rank $\le n-1$. It is easy to see that any subgroup $G$ of ${\rm Aut}(X)$ can be considered as a subgroup of $GL(V, {\bf R})$, where $V$ is the Dolbeault cohomology group $H^{1,1}(X, {\bf R})$.

We call a group \emph{virtually solvable} if it has a solvable subgroup of finite index. Let $V_{\bf C}$ be a finite dimensional vector space over ${\bf C}$. A solvable subgroup $G$ of $GL(V_{\bf C})$ is called \emph{connected} if its Zariski closure $\bar G$ in $GL(V_{\bf C})$ is connected. It can be shown that given a virtually solvable subgroup $G$ of $GL(V_{\bf C})$, we can find a connected solvable, finite index subgroup $G_1$ of $G$. Moreover, it is easy to see that any subgroup of a solvable group $G$ and any quotient group of $G$ are solvable and that the closure of $G$ is also solvable (see Section 2 of \cite{KOZ}). For a subgroup of $GL(V_{\bf C})$, Tits proved the following alternative theorem in \cite{Ti}.

\begin{theorem}[Tits] \label{thm1.1}
Let $G$ be a subgroup of $GL(V_{\bf C})$. Then $G$ is either virtually solvable or contains a non-commutative free group ${\bf Z}\ast {\bf Z}$.
\end{theorem}

In view of the above Tits' alternative and Dinh-Sibony's theorem, it is natural to ask the following interesting conjecture which is called a \emph{conjecture of Tits type}.

\begin{conjecture} \label{conj1.1}
Let $X$ be a compact K\" ahler manifold of complex dimension $n$, and let $G$ be a connected solvable subgroup of the automorphism group ${\rm Aut}(X)$. Let
\[
N(G)=\{ g\in G\, |\, g\ \text{is\ of\ null\ entropy} \}.
\]
Then $G/N(G)$ is a free abelian group of rank $\le n-1$.
\end{conjecture}
It is clear that the rank estimate is optimal from the case $X=E^n$, where $E$ is an elliptic curve.

The aim of this paper is to give a proof of the conjecture of Tits type and show some more non-trivial results. Indeed, the conjecture has been proved  by Keum, Oguiso, and Zhang in the paper \cite{KOZ} except for the rank estimate and then in full generality by Zhang in \cite{Zh}. It seems to have been known that the Tits alternative modulo the rank estimate is true by a direct consequence of Tits theorem for linear groups and Lieberman's results on automorphism groups (see \cite{Ca}). In \cite{KOZ}, Keum, Oguiso, and Zhang also exhibited various examples such as complex tori, hyperk\" ahler manifolds and minimal threefolds for which the full conjecture of Tits type holds.

The proof of Theorem \ref{thm1.1} of this paper gives the sharp rank estimate as well as some non-obvious implications on the structure of solvable automorphism groups of compact K\" ahler manifolds. We like to emphasize that Theorem \ref{thm1.2} (b)-(c) below are completely new. The main strategy of the proof comes from the paper \cite{DS} of Dinh and Sibony and the theorem of Birkhoff-Perron-Frobenius (or the theorem of Lie-Kolchin type in \cite{KOZ}). Our first main result is

\begin{theorem} \label{thm1.2}
Let $X$ be a compact K\" ahler manifold of complex dimension $n$, and let $G$ be a connected solvable subgroup of the automorphism group ${\rm Aut}(X)$. Let
\[
N(G)=\{ g\in G\, |\, g\ \text{is\ of\ null\ entropy} \}.
\]
Then the following properties hold:
\begin{itemize}
\item[\rm (a)] $G/N(G)$ is a free abelian group of rank $r(G)\le n-1$. Furthermore, the rank estimate is optimal.
\item[\rm (b)] Let $h_k$ be the real dimension of the cohomology group ${H^{k,k}(X, {\bf R}})$. If $r(G)= n-1$, then $h_k$ satisfies
    \begin{equation}\label{eq1.1}
    h_k\ge \binom{n-1}{k},\quad 1\le k\le n-1.
    \end{equation}
    In addition, if $k$ divides $n-1$, then the lower bound of \eqref{eq1.1} can be improved by one, i.e., $h_k\ge \binom{n-1}{k}+1$.

\item[\rm (c)] Let $\bar{\mathcal K}$ be the closure of the K\" ahler cone $\mathcal{K}$. Then there exist $(r(G)+1)$ many non-zero classes $c_1, \ldots, c_{r(G)+1}$ in $\bar{\mathcal K}$ such that
    \[
    c_1\wedge c_2\wedge \cdots \wedge c_{r(G)+1}\ne 0.
    \]
\end{itemize}
\end{theorem}

It is clear that this theorem is a generalization of the result of Dinh and Sibony for the abelian automorphisms with positive entropy to the solvable automorphisms. The proofs of Theorem \ref{thm1.2} (a) and (b)--(c) will be given in Theorem \ref{thm3.1} and Proposition \ref{prop3.1} of Section \ref{sec3}, respectively.

In fact, in \cite{DS} Dinh and Sibony also proved that in case that $G$ is abelian and the rank $r(G)$ is equal to $n-1$, $N(G)$ is finite (Proposition 4.7 in \cite{DS}). Let ${\rm Aut}_0(X)$ denote the identity component of ${\rm Aut}(X)$ consisting of automorphisms homotopically equivalent to the identity. In a recent paper \cite{Zh10}, D.-Q. Zhang investigated the question of the finiteness of $N(G)$ for solvable automorphism groups $G$, and proved that if $r(G)=n-1=2$ and ${\rm Aut}_0(X)$ is trivial, then $N(G)$ is finite (Theorem 1.1 (3) in \cite{Zh10}). In this paper, we also consider the question of the finiteness of $N(G)$ for solvable automorphism groups $G$ and, as a consequence of Theorem \ref{thm1.1}, significantly extend the result by Zhang which holds only for the complex dimension equal to $3$. To be precise, our another main result of this paper which affirmatively answers to Question 2.18 in \cite{Zh} is:

\begin{theorem} \label{thm1.3}
Let $X$ be a compact K\" ahler manifold of complex dimension $n$, and let $G$ be a connected solvable subgroup of the automorphism group ${\rm Aut}(X)$. Assume that the rank $r(G)$ of the quotient group $G/N(G)$ is equal to $n-1$ and that ${\rm Aut}_0(X)$ is trivial. Then $N(G)$ is a finite set.
\end{theorem}

If $G$ is further assumed to be abelian in Theorem \ref{thm1.3}, it was shown in \cite{DS} that, even without the triviality of ${\rm Aut}_0(X)$, $N(G)$ is finite (Proposition 2.2 in \cite{Li}). However, it is known as in Example 4.5 of \cite{DS} that there is an abelian variety $X$ of complex dimension $n$ with an automorphism group $G$ such that $N(G)={\rm Aut}_0(X)\cong X$ and the rank $r(G)=n-1$.

We organize this paper as follows. In Section \ref{sec2}, we construct a homomorphism from the solvable subgroup of automorphisms to the abelian group ${\bf R}^m$. Here one of the key ingredients is the theorem of Birkhoff-Perron-Frobenius in \cite{Bir} (or Lie-Kolchin type in \cite{KOZ}). In Section \ref{sec3}, we give a detailed proof of Theorem \ref{thm1.2}.
Finally, we give a proof of Theorem \ref{thm1.3} in Section \ref{sec4}.

\section{Theorem of Birkhoff-Perron-Frobenius and its Applications} \label{sec2} %%section 2

The goal of this section is to set up some preliminary results to prove the main Theorem \ref{thm1.2}. One of the key ingredients is the theorem of Birkhoff-Perron-Frobenius in \cite{Bir} (or more generally the theorem of Lie-Kolchin type established in \cite{KOZ}).

Let $G$ be a solvable group and $\rho: G\to GL(V_{\bf C})$ be a complex linear representation of $G$. Let $Z$ be the Zariski closure of $\rho(G)$ in $GL(V_{\bf C})$, and let $Z_0$ be the connected component of the identity in $Z$. Let $G_0=\rho^{-1}(Z_0)$. Then the group $Z_0$ is conjugate to a group of upper triangular matrices whose determinant is non-zero. Let $N(G_0)$ be the subgroup of $G_0$ whose elements are defined by the statement that $f$ is an element of $N(G_0)$ if and only if all eigenvalues of $f^\ast$ on $V_{\bf C}$ are equal to $1$. Then $N(G_0)$ is a normal subgroup of $G_0$ and the abelian group $G_0/N(G_0)$ embeds into $({\bf C}^\ast)^{\dim V_{\bf C}}$. We shall denote by $G$ the group $G_0$, from now on.

Then we have the following lemma whose proof is simple (e.g., see Lemma 2.5 in \cite{KOZ}).

\begin{lemma} \label{lem2.1}
Let $Z_0$ be a connected solvable subgroup of $GL(V_{\bf C})$. Then the eigenvalues of every element of the commutator subgroup $[Z_0,Z_0]$ of $Z_0$ are all equal to $1$.
\end{lemma}

Since $G$ is solvable, there exists a derived series of $G$ as follows.
\[
G=G^{(0)} \triangleright G^{(1)} \triangleright G^{(2)} \triangleright \ \  \cdots \ \ \triangleright G^{(k)} \triangleright G^{(k+1)}=\{ {\rm id} \},
\]
where $G^{(i+1)}$ is a normal subgroup of $G^{(i)}$ and $G^{(i+1)}$ is the commutator subgroup $[G^{(i)}, G^{(i)}]$ of $G^{(i)}$ ($0\le i \le k$).  Let $A=G^{(k)}$. Then $A$ is an abelian subgroup of $G$, and clearly $A$ is a subset of $[G,G]$. So all the eigenvalues of $A$ are also equal to $1$ by Lemma \ref{lem2.1}.

Recall that if ${C}$ is a subset of a real vector space $V$, then  ${C}$ is said to be a \emph{strictly convex closed cone} of $V$ if $C$ is closed in $V$, closed under addition and multiplication by a non-negative scalar, and contains no 1-dimensional linear space. We then will need the following theorem of Birkhoff-Perron-Frobenius in \cite{Bir}.

\begin{theorem} \label{thm2.1}
Let $C$ be a nontrivial strictly convex closed cone of $V$ with non-empty interior in $V$. Then any element $g$ of $GL(V)$ such that $g(C)\subset C$ has an eigenvector $v_g$ in $C$ whose eigenvalue is the spectral radius of $g$ in $V$.
\end{theorem}

In fact, if we make use of the subgroup $A$ of $[G,G]$ and Lemma \ref{lem2.1}, one can obtain a stronger version for connected solvable subgroups of $GL(V)$ as in \cite{KOZ} which is called the \emph{theorem of Lie-Kolchin type for a cone}. It plays an important role in the present paper. For more precise statement of Theorem \ref{thm2.2}, see Theorem 2.1 in \cite{KOZ}.

\begin{theorem} \label{thm2.2}
Let $V$ be an $r$-dimensional real vector space, and let $C\ne \{ 0 \}$ be a strictly convex closed cone of $V$. Let $G$ be a connected solvable subgroup of $GL(V)$ such that $G(C)\subset C$. Then there is a nonzero vector $v\in C\backslash\{ 0 \}$ which is invariant under $G$.
\end{theorem}

In this paper we shall apply the above discussion to a solvable subgroup $G$ of $\text{Aut}(X)$ acting on
\[
V_{\bf C}=H^{1,1}(X, {\bf R})\otimes {\bf C}.
\]
Then, in particular, we easily see that all the eigenvalues of $[G,G]$ are equal to $1$ (i.e., every element of $[G,G]$ is \emph{unipotent}). Since $A$ is a subset of $[G,G]$, it is also obvious that all the eigenvalues of every element of $A$ are equal to $1$ and so every element of $A$ is of null entropy.

In order to state and prove Corollary \ref{cor2.1}, we need to introduce more notations. Namely, when $X$ is a compact connected K\" ahler manifold of complex dimension $n$ as before, let ${\mathcal K}$ denote the K\" ahler cone in the Dolbeault cohomology group $H^{1,1}(X, {\bf R})$. So ${\mathcal K}$ is the cone of strictly positive smooth $(1,1)$-forms in $H^{1,1}(X, {\bf R})$, and it is a strictly convex open cone in $H^{1,1}(X, {\bf R})$ whose closure $\bar{\mathcal K}$ is also a strictly convex closed cone such that $\bar{\mathcal K}\cap - \bar{\mathcal K}=\{ 0 \}$.

With these understood, we can obtain the following easy corollary which will play an important role in Theorem \ref{thm3.1}.

\begin{cor} \label{cor2.1}
Let $G$ be a connected solvable group of automorphisms of a compact K\" ahler manifold. Assume that $G$ contains an element $f$ of positive entropy. Then there exist a non-zero class $c$ in $\bar{\mathcal K}$ invariant under $G$ and a positive real number $\chi(f)$ so that
\begin{equation*}
f^\ast (c)= \chi(f) c\quad \text{and}\quad \rho(f)\ge \chi(f).
\end{equation*}
\end{cor}

\begin{proof}
To prove it, we simply take ${C}=\bar{\mathcal K}$. Then  it follows from Theorem \ref{thm2.2} of Lie-Kolchin type (or Theorem 2.1 of \cite{KOZ}) that there exists a common non-zero eigenvector $c\in {C}= \bar{\mathcal K}$ for $G$. Hence we can write
\[
f^\ast(c)=\chi(f) c,\quad \chi(f)\in {\bf R}.
\]
So we are done.
\end{proof}

To give a proof of Theorem \ref{thm1.2}, we need to set up some more terminology.

\begin{defn}
Let $\tau=( \tau(f) )_{f\in G}\in {\bf R}^{G}$, and let $\Gamma_\tau$ be the cone of classes $c$ in $\bar{\mathcal K}$ such that
\[
f^\ast (c)= \exp( \tau(f) ) c
\]
for all $f\in G$.
\end{defn}

Let
\[
F=\{ \tau\in {\bf R}^G\, |\, \Gamma_\tau \ne \{ 0 \}\}.
\]
If $\Gamma_\tau\ne \{0 \}$ then $\exp( \tau(f) )$ is an eigenvalue of $f^\ast$ on $V=H^{1,1}(X, {\bf R})$. Since $V$ is finite dimensional, $F$ must be finite. Let $\tau_1, \tau_2, \cdots, \tau_m$ be all the elements of a finite set $F$. If we define a map
\begin{equation*}
\pi: G \to {\bf R}^m,\quad f\mapsto (\tau_1(f), \tau_2(f),\ldots, \tau_m(f)),
\end{equation*}
then we can show the following lemma:

\begin{lemma} \label{lem2.2}
\begin{itemize}
\item[\rm (a)] The integer $m$ satisfies the inequality
\[
m\le h_1=\dim H^{1,1}(X, {\bf R}).
\]

\item[\rm (b)] The map $\pi$ is always a homomorphism into the abelian group $({\bf R}^m, +)$. In particular, the image $\pi(G)$ is also abelian.
\end{itemize}
\end{lemma}

\begin{proof}
For the proof of (a), note that, since $\tau_1, \tau_2, \cdots, \tau_m$ are all distinct, there exists an element $f_0\in G$ such that $\tau_i(f_0)\ne \tau_j(f_0)$ for $1\le i< j\le m$. Thus $f^\ast_0$ on $V$ has at most $m$ distinct eigenvalues. Hence $m$ is less than or equal to the dimension of $V$ that is $h_1$.

For the proof of (b), it suffices to prove that $\pi$ is a homomorphism. To do so, for $c_i\in \Gamma_{\tau_i}$ note that
\[
(f\circ g)^\ast c_i = \exp (\tau_i(f) + \tau_i(g)) c_i.
\]
This implies that $\pi(f\circ g)= \pi(f)+\pi(g)$, so $\pi$ is a group homomorphism. This completes the proof of Lemma \ref{lem2.2}.
\end{proof}

\section{Rank Estimate: Proof of Theorem \ref{thm1.2}} \label{sec3} %%section 3

In this section we give a proof of Theorem \ref{thm1.2}. Indeed, the proof of this section is essentially an adaptation of the proof of the Principal Theorem by Dinh and Sibony in \cite{DS}. For the sake of reader's convenience, we shall show how to prove Theorem \ref{thm1.2} relatively in detail. See \cite{DS} for more details.

First we need some preliminary lemma from the paper of Dinh and Sibony in \cite{DS}. Assume that $X$ is a compact K\" ahler manifold of dimension $n$, as before.

\begin{lemma} \label{lem3.1}
Let $c, c'$, and $c_i$ be the non-zero classes in $\bar{\mathcal K}$, $1\le i \le t\le n-2$. Assume that for $f\in \text{\rm Aut}(X)$ there exist two different positive real constants $\lambda$ and $\lambda'$ such that
\begin{equation*}
\begin{split}
f^\ast(c_1\wedge \cdots \wedge c_t\wedge c) &=\lambda c_1 \wedge \cdots \wedge c_t \wedge c,\\
f^\ast(c_1\wedge \cdots \wedge c_t\wedge c') &=\lambda' c_1 \wedge \cdots \wedge c_t \wedge c'.
\end{split}
\end{equation*}
Assume also that $c_1 \wedge \cdots \wedge c_t \wedge c\ne 0$ and $c_1 \wedge \cdots \wedge c_t \wedge c\wedge c'=0$. Then we have
\[
c_1 \wedge \cdots \wedge c_t \wedge c'=0.
\]
\end{lemma}

\begin{proof}
This lemma has nothing to do with a solvable subgroup of automorphisms of $X$ and will play an essential role in the proof of Lemma \ref{lem3.2} and Theorem \ref{thm3.1}. See Lemma 4.3 of \cite{DS}.
\end{proof}

As the proof of Theorem \ref{thm3.1} below shows, if the rank $\tilde r$ of the image of the homomorphism $\pi$ is greater than or equal to $n-1$, it will be enough to use the homomorphism $\pi$ in order to prove Theorem \ref{thm1.2}. However, it turns out that if $\tilde r$ is less than or equal to $n-2$, we need some more homomorphisms from $G$ to ${\bf R}$ other than $\tau_i$'s, in order to obtain an injective homomorphism from $G/N(G)$. The following lemma provides such additional homomorphisms. Here we adapt a variation of some arguments originated from \cite{DS} (see also \cite{Zh}).

\begin{lemma} \label{lem3.2}
Let $\tilde r$ denote the rank of the image of the homomorphism $\pi$. Then the following properties hold:

\begin{itemize}
\item[\rm (a)] There exist non-zero classes $c_j$ $(j=1,2,\cdots, \tilde r)$ in $\bar{\mathcal K}$, invariant under $G$, such that
\[
c_1\wedge c_2\wedge \cdots \wedge c_{\tilde r}\ne 0.
\]

\item[\rm (b)] Assume that $\tilde r\le n-2$. Then there exist additional non-zero classes $c_{\tilde r+j}$ in $\bar{\mathcal K}$, not necessarily invariant under $G$, and homomorphisms $\tilde\tau_{\tilde r+j}:G\to {\bf R}$ $(j=1,\cdots, n-\tilde r-1)$ such that
    \[
    c_1\wedge c_2\wedge\cdots\wedge c_{n-2}\wedge c_{n-1}\ne 0
    \]
    and such that for all $f\in G$
    \begin{equation} \label{eq3.1}
    \begin{split}
    & f^\ast (c_1\wedge c_2\wedge \cdots \wedge c_{\tilde r}\wedge c_{\tilde r+1}\wedge \cdots \wedge c_{\tilde r +j})\\
    &= \exp({\tau_1(f)})\cdots \exp({\tau_{\tilde r}(f)})\exp({\tilde\tau_{\tilde r+1}(f)})\cdots \exp({\tilde\tau_{\tilde r +j}(f)})\\
    &\cdot c_1\wedge c_2\wedge \cdots \wedge c_{\tilde r}\wedge c_{\tilde r+1}\wedge \cdots \wedge c_{\tilde r +j}.
    \end{split}
    \end{equation}
\end{itemize}
\end{lemma}

\begin{proof}
To prove (a), we assume without loss of generality that the first $\tilde r$ coordinates of the map $\pi$ generates the image of the map $\pi$, and let us denote by $\tau_1,\cdots, \tau_{\tilde r}$ such $\tilde r$ coordinates. Let $c_i$ be a non-zero class in $\Gamma_{c_i}$ $(1\le i\le \tilde r)$. Then for any $I=\{ i_1, \cdots, i_k \}\subset \{ 1,2,\cdots, \tilde r \}$, let
\[
 c_I=c_{i_1}\wedge \cdots \wedge c_{i_{k-1}}\wedge c_{i_k}.
\]
Then it suffices to show that $c_I\ne 0$ for any subset $I$ of $\{ 1,2,\cdots, \tilde r \}$. To show it, we let $I=\{1,2,\cdots, k\}$ for simplicity. The case of $k=1$ is trivial. So suppose that
\begin{equation*}
\begin{split}
&c_1\wedge \cdots \wedge c_{k-2}\wedge c_{k-1}\ne 0, \quad c_1\wedge \cdots \wedge c_{k-2}\wedge c_k\ne 0,\\
& c_1\wedge \cdots \wedge c_{k-1}\wedge c_k=0.
\end{split}
\end{equation*}
If we apply Lemma \ref{lem3.1} for $t=k-2$, $c=c_{k-1}$, and $c'=c_k$, then it is easy to see that $\tau_{k-1}(f)=\tau_k (f)$ for all $f\in G$. This implies that the image of $\pi$ lies in the hyperplane $\{ x_{k-1}=x_{k} \}$, which contradicts the choice of $\tau_i$'s. This completes the proof of (1).

For (b), we continue to use the notations used in the proof of (a). Then consider the induced action of $G$ on the real subspace $c_1\wedge\cdots\wedge c_{\tilde r}\wedge H^{1,1}(X, {\bf R})$ of $H^{\tilde r+1,\tilde r+1}(X, {\bf R})$, where $c_1, \cdots, c_{\tilde r}$ are the non-zero classes invariant under $G$ which are obtained in (a) above. Then $c_1\wedge\cdots\wedge c_{\tilde r}\wedge \bar{\mathcal K}$ spans $c_1\wedge\cdots\wedge c_r\wedge H^{1,1}(X, {\bf R})$. Moreover, by using the dual homology classes of $c_i$'s $(i=1,2\cdots, {\tilde r})$ under the natural pairing and the slant product operation (e.g., see p. 286 of \cite{Sp}), one can show that  $c_1\wedge\cdots\wedge c_{\tilde r}\wedge \bar{\mathcal K}$ is a strictly convex closed cone in $c_1\wedge\cdots\wedge c_r\wedge H^{1,1}(X, {\bf R})$ which is invariant under $G$. So, if we apply Theorem \ref{thm2.2} of Lie-Kolchin type to $c_1\wedge\cdots\wedge c_{\tilde r}\wedge \bar{\mathcal K}$ in $c_1\wedge\cdots\wedge c_{\tilde r}\wedge H^{1,1}(X, {\bf R})$, we obtain a non-zero class $c_{\tilde r+1}$ in $\bar{\mathcal K}$ such that
\begin{equation*}
\begin{split}
&f^\ast(c_1\wedge c_2\wedge \cdots \wedge c_{\tilde r+1})\\
&=\exp(\tau_{1}(f))\exp(\tau_{2}(f))\cdots \exp(\tilde\tau_{\tilde r+1}(f)) c_1\wedge c_2\wedge \cdots \wedge c_{\tilde r+1}
\end{split}
\end{equation*}
for some functions $\tilde\tau_{\tilde r+1}:G\to {\bf R}$. Next, make use of the mathematical induction to obtain the rest of the non-zero classes $c_{\tilde r+j}\in \bar{\mathcal K}$ satisfying the equation \eqref{eq3.1}. As in Lemma \ref{lem2.2}, it is straightforward to show that each $\tilde\tau_{\tilde r+j}$ $(j=1,2,\cdots, n-\tilde r-1)$ is a homomorphism. So we are done.
\end{proof}

With these preliminaries, we give a definition of the homomorphism $\Pi$ which will be used in Theorem \ref{thm3.1}.

\begin{defn}
Let $\tilde r$ denote the rank of the image of the homomorphism $\pi$. From now on until the end of this section, we assume without loss of generality that the first $\tilde r$ coordinates of the map $\pi$ generates the image of the map $\pi$, and let us denote by $\tau_1,\cdots, \tau_{\tilde r}$ such $\tilde r$ coordinates.

\begin{itemize}
\item[\rm (a)] If $\tilde r\ge n-1$, a homomorphism $\Pi: G\to {\bf R}^{n-1}$ is just defined to be the homomorphism $\pi$ composed by the canonical projection onto the first $n-1$ factors.

\item[\rm (b)] On the other hand, if $1\le \tilde r\le n-2$, we define a homomorphism
\[
\Pi: G\to {\bf R}^{n-1},\quad f\mapsto (\tau_1(f),\cdots, \tau_{\tilde r}(f), \tilde\tau_{\tilde r+1}(f),\cdots, \tilde \tau_{n-1}(f)).
\]
\end{itemize}
\end{defn}

We then are ready to prove Theorem \ref{thm3.1}, which clearly gives a proof of Theorem \ref{thm1.2} as a byproduct.

\begin{theorem} \label{thm3.1}
Let $G$ be a connected solvable group of automorphisms of a compact K\" ahler manifold $X$ of complex dimension $n$, and let $N(G)$ be the normal subgroup of $G$ defined as in {\rm Theorem \ref{thm1.2}}. Then the induced homomorphism $\Pi: G/N(G)\to {\bf R}^{n-1}$ is injective and its image is discrete. In particular, $G/N(G)$ is a free abelian group of rank $\le n-1$.
\end{theorem}

\begin{proof}
In order to prove the proposition, we shall first show that $\Pi(G/N(G))$ is discrete in the additive group ${\bf R}^{n-1}$ with the standard topology. To do so, it suffices to show that $(0)_{i=1}^{n-1}$ is an isolated point in the image of $\Pi$ with the induced topology from ${\bf R}^{n-1}$, since the map $\Pi$ is a homomorphism.

We first deal with the case of $\tilde r\ge n-1$. To do so, let $\delta$ be an arbitrary positive real number. Then consider the set of all elements $f\in G/N(G)$ satisfying the following condition:
\begin{equation} \label{eq3.2}
|\tau_j(f)| < \delta, \quad j=1,2,\cdots, n-1.
\end{equation}
By Theorem \ref{thm2.1} of Birkhoff-Perron-Frobenius, there exist non-zero classes $\kappa_i$ ($i=1,2$) in the closure of the K\" ahler cone such that $f^\ast(\kappa_i)=\lambda_i \kappa_i$ with $\lambda_1=\rho(f)>1$ and $\lambda_2^{-1}=\rho(f^{-1})>1$. Moreover,  it follows from Lemma \ref{lem3.2} (a) that there exist non-zero classes $c_{j}\in \bar{\mathcal K}$ ($j=1,2,\cdots, n-1$) such that
\[
c_1\wedge c_2\wedge \cdots \wedge c_{n-1}\ne 0\ \text{and}\ f^\ast(c_{j})=\exp(\tau_{j}(f))c_{j}
\]
for all $f\in G$. Then we have
\begin{equation*}
\begin{split}
&c_1\wedge c_2\wedge \cdots \wedge c_{n-1}\wedge \kappa_i = f^\ast(c_1\wedge c_2\wedge \cdots\wedge c_{n-1}\wedge \kappa_i)\\
&=\exp(\tau_1(f))\cdots \exp(\tau_{n-1}(f))\lambda_i (c_1\wedge c_2\wedge \cdots \wedge c_{n-1}\wedge \kappa_i),
\end{split}
\end{equation*}
where in the first equality we used the fact $\text{deg}(f)=1$.

If $c_1\wedge c_2\wedge \cdots \wedge c_{n-1} \wedge \kappa_i\ne 0$ for all $i=1, 2$, then we have
\[
\exp(\tau_1(f))\cdots \exp(\tau_{n-1}(f))\lambda_i=1.
\]
Thus $1 < \lambda_1=\lambda_2 <1$, which is a contradiction.

Next, we assume that $c_1\wedge c_2\wedge \cdots \wedge c_{n-1}\wedge \kappa_1= 0$ by changing the role of $f^{-1}$ by $f$ and vice versa, if necessary. If $c_1 \wedge \kappa_1=0$ then it follows from Corollary 3.2 of \cite{DS} that $c_1$ is parallel to $\kappa_1$, and so we have
\[
\rho(f)=\lambda_1=\exp(\tau_1(f)) < e^\delta.
\]
On the other hand, if $c_1\wedge \kappa_1\ne 0$, then we let $l$ be the minimal integer $\ge 2$ such that $c_1\wedge c_2\wedge\cdots \wedge c_l\wedge \kappa_1=0$. Then it follows from the definition of $l$ that $c_1\wedge c_2\wedge\cdots \wedge c_{l-1}\wedge \kappa_1\ne 0$, and so clearly $c_1\wedge c_2\wedge\cdots \wedge c_{l-2}\wedge c_{l-1}\ne 0$ and $c_1\wedge c_2\wedge\cdots \wedge c_{l-2}\wedge \kappa_1\ne 0$. Furthermore, by its construction and $l\le n-1\le \tilde r$, we have
\[
c_1\wedge c_2\wedge \cdots \wedge c_{l-1}\wedge c_l\ne 0.
\]
It is also obvious that the following two identities hold:
\begin{equation} \label{eq3.3}
\begin{split}
&f^\ast(c_1\wedge \cdots \wedge c_{l-1}\wedge \kappa_1) \\
&= \exp(\tau_1(f))\exp(\tau_{2}(f))\cdots \exp(\tau_{l-1}(f)) \lambda_1 c_1 \wedge \cdots \wedge c_{l-1}\wedge \kappa_1,\\
&f^\ast(c_1\wedge \cdots \wedge c_{l-1}\wedge c_{l}) \\
&= \exp(\tau_1(f))\exp(\tau_{2}(f))\cdots \exp(\tau_{l}(f)) c_1 \wedge \cdots \wedge c_{l-1} \wedge c_{l}.
\end{split}
\end{equation}
Now, once again applying Lemma \ref{lem3.1} to the equations in \eqref{eq3.3} with the roles of $t=l-1$, $c=\kappa_1$ and $c'=c_{l}$, we have
\begin{equation*}
\begin{split}
&\exp(\tau_1(f))\exp(\tau_{2}(f))\cdots \exp(\tau_{l-1}(f))\lambda_1\\
&=\exp(\tau_1(f))\exp(\tau_{2}(f))\cdots \exp(\tau_{l-1}(f))\exp(\tau_{l}(f)).
\end{split}
\end{equation*}
Thus we obtain
\begin{equation*}
\lambda_1=\rho(f)=\exp(\tau_{l}(f)) < e^\delta.
\end{equation*}
Since $\rho(f^{\pm 1})$ is always less than or equal to $\rho(f^{\mp})^{n-1}$, we also have $\rho(f^{-1})\le \exp({(n-1)\delta})$ by the equation \eqref{eq3.3}. Hence, for all $f\in G$ satisfying the inequality \eqref{eq3.2} the absolute values of all the eigenvalues of $f^\ast|_{H^{1,1}(X, {\bf C})}$ are bounded by $\exp((n-1)\delta)$.

Now let $\Psi_f(x)$ be the characteristic polynomial of $f^\ast$ on $H^{1,1}(X, {\bf C})$. Then $\Psi_f(x)$ can be assumed to be a polynomial with integer coefficients, since $G$ can be regarded as a subgroup of $\text{GL}(H^2(X, {\bf Z}))$. Since the absolute values of all the eigenvalues of $f^\ast|_{H^{1,1}(X, {\bf C})}$ of $f$ satisfying \eqref{eq3.2} are shown to be bounded by $\exp((n-1)\delta)$, the coefficients of the characteristic polynomials $\Psi_f(x)$ of all such $f^\ast$'s are all bounded. This implies that there are only finitely many such characteristic polynomials and the set of all such vectors $\pi(f)$ is finite. Therefore the zero vector $(0)_{i=1}^{n-1}$ is isolated in $\pi(G/N(G))$ and so the image $\pi(G/N(G))$ is discrete.

Next assume that there is an element $f$ of positive entropy with $\pi(f)=(0)_{i=1}^{n-1}\in {\bf R}^{n-1}$. Then it follows from the exactly same arguments as above that there exists a $\tau_l(f)$ for some $1\le l\le n-1$ such that $\rho(f)=\exp(\tau_l(f))>1$. But, this is obviously a contradiction. This implies that the kernel of the map $\pi$ is just $N(G)$. So the induced map $\pi: G/N(G)\to {\bf R}^{n-1}$ is actually injective.  For case of $\tilde r\ge n-1$, this completes the proof that $G/N(G)$ is free abelian of rank $n-1$ and that $\tilde r=n-1$.

The proof of other case of $\tilde r\le n-2$ is completely similar. But this time we need to use $\tau_1$, $\cdots$, $\tau_{\tilde r}$, $\tilde\tau_{\tilde r +1}$, $\cdots$, $\tilde\tau_{n-1}$ instead of $\tau_1$, $\cdots$, $\tau_{n-1}$. We leave its detailed proof to a reader. This completes the proof of Theorem \ref{thm3.1}.
\end{proof}

Note that the proof of Theorem \ref{thm3.1} shows that if $\tilde r\ge n-1$, then $G/N(G)$ is a free abelian group of rank $n-1$. Conversely, if $G/N(G)$ is a free abelian group of rank $n-1$, then $\tilde r\ge n-1$, since the map $\Pi$ is injective. On the other hand, if $\tilde r \le n-2$, then we can just say that $G/N(G)$ is a free abelian group of rank $\le n-1$. As another interesting consequence of the injectivity of $\Pi$ in Theorem \ref{thm3.1}, one can prove the following proposition which is straightforward from Proposition 4.4 of \cite{DS}.

\begin{prop} \label{prop3.1}
Let $G$ be a connected solvable subgroup of the automorphism group ${\rm Aut}(X)$ and let $N(G)$ be the normal subgroup of null entropy, as in {\rm Theorem \ref{thm1.2}}. Let $r$ denote the rank of the quotient group $G/N(G)$. Then the following properties hold:

\begin{itemize}
\item[\rm (a)] Let $h_k$ be the real dimension of the cohomology group ${H^{k,k}(X, {\bf R}})$. Assume that $r= n-1$. Then $h_k$ satisfies
    \[
    h_k\ge \binom{n-1}{k},\quad 1\le k\le n-1.
    \]
    In addition, if $k$ divides $n-1$ as well, then $h_k\ge \binom{n-1}{k}+1$.

\item[\rm (b)] There exist $(r+1)$ non-zero classes $c_1, \ldots, c_{r+1}$ in $\bar{\mathcal K}$ such that
    \[
    c_1\wedge c_2\wedge \cdots \wedge c_{r+1}\ne 0.
    \]
\end{itemize}
\end{prop}

\begin{proof}
It suffices to notice that the proof of Proposition 4.4 of \cite{DS} works also for solvable subgroups of automorpthisms. One of the main reasons is that there exists an induced homomorphism $\Pi: G/N(G)\to {\bf R}^{n-1}$ which is injective by Theorem \ref{thm3.1}. To be precise, for the proof we need to use Theorem \ref{thm2.1} instead of Lemma 4.1 in \cite{DS} as well as Lemma \ref{lem3.1}. But this does not make any significant difference in the proof, since in any case $\Pi(G/N(G))$ cannot be contained in a finite union of hyperplanes in ${\bf R}^r$. This completes the proof of Proposition \ref{prop3.1}.
\end{proof}

In particular, Proposition \ref{prop3.1} (b) implies that $r+1$ is less than or equal to $n$, due to the dimensional reason of $X$. Therefore, $r$ should be less than or equal to $n-1$. Once again, this proves Theorem \ref{thm1.2}.

\section{Proof of Theorem \ref{thm1.3}} \label{sec4} %section 4

The aim of this section is to give a proof of Theorem \ref{thm1.3}. To do so, assume that the rank $r(G)$ of $G/N(G)$ is equal to $n-1$. Then it follows from Theorem \ref{thm1.1} and the homomorphism $\Pi$ constructed in Section \ref{sec3} that there exist non-zero classes $c_1,\cdots, c_n$ in the closure $\bar{\mathcal K}$ of the K\" ahler cone ${\mathcal K}$ such that
\begin{itemize}
 \item $f^\ast c_i = c_i$ for all $1\le i\le n$ and all $f\in N(G)$.
 \item $c_1\wedge c_2\wedge \cdots \wedge c_n\ne 0$.
\end{itemize}
Now, let $c=c_1+c_2\cdots+c_n$. Then clearly $f^\ast(c)=c$ for all $f\in N(G)$. Since $c_1\wedge c_2\cdots\wedge c_n\ne 0$, it is also clear that $c^n\ne 0$. It is known by a theorem of Demailly and Paun in \cite{DP} that the K\" ahler cone ${\mathcal K}$ is connected and that every class of ${\mathcal K}$ satisfies $\int_X c^n >0$. Since $c$ lies in $\bar{\mathcal K}$ and $c^n\ne 0$, it follows that $c$ lies in the K\" ahler cone ${\mathcal K}$ and so $c$ is a K\" ahler class.

Now, let ${\rm Aut}_c(X)$ denote the automorphism group preserving the K\" ahler class $c$. Then $N(G)$ is a subgroup of ${\rm Aut}_c(X)$. Recall that the quotient group ${\rm Aut}_c(X)/{\rm Aut}_0(X)$ is a finite group by a theorem of Liberman (Proposition 2.2 in \cite{Li}). Hence, if ${\rm Aut}_0(X)$ is trivial, then ${\rm Aut}_c(X)$ is finite. Therefore $N(G)$ is a finite set, which completes the proof of Theorem \ref{thm1.3}.

%%\smallskip\smallskip
%%\noindent{\bf Acknowledgements:}
%%The author is grateful to a reader for his remarkable list of corrections.
%%This work was supported by the Korea Research Foundation Grant funded by the Korean Government (MOEHRD, Basic Research Promotion Fund) (KRF-2007-312-C00025). It was also partially supported by the SRC Program of KOSEF.
%%\smallskip\smallskip

%%%%------------------------------------------

\end{document}